\documentclass[12pt]{amsart}

\usepackage{graphicx}
\usepackage{amsthm}
\usepackage{amsmath, amssymb}
\swapnumbers

\theoremstyle{plain}
\newtheorem{Thm}{Theorem}

\newtheorem{Lem}[Thm]{Lemma}

\theoremstyle{definition}
\newtheorem{Def}[Thm]{Definition}

\newcommand{\comment}[1]{}

\begin{document}

\title[Automorphisms of the three-torus]{Automorphisms of the three-torus preserving a genus three Heegaard splitting.}

\author{Jesse Johnson}
\address{\hskip-\parindent
        Department of Mathematics \\
        Yale University \\
        PO Box 208283 \\
        New Haven, CT 06520 \\
        USA}
\email{jessee.johnson@yale.edu}


\thanks{Research supported by NSF MSPRF grant 0602368}

\begin{abstract}
The mapping class group of a Heegaard splitting is the group of connected components in the set of automorphisms of the ambient manifold that map the Heegaard surface onto itself.  For the genus three Heegaard splitting of the 3-torus, we find an eight element generating set for this group.  Six of these generators induce generating elements of the mapping class group of the 3-torus and the remaining two are isotopy trivial in the 3-torus.  
\end{abstract}

\maketitle

\section{Introduction}
Given a 3-manifold $M$ and a Heegaard splitting $(\Sigma, H_1, H_2)$ of $M$, consider the set $Aut(M)$ of orientation preserving automorphisms $M \rightarrow M$.  The set of connected components of $Aut(M)$ forms a group $Mod(M)$ called the mapping class group.  We will define $Aut(M, \Sigma)$ to be the subset of $Aut(M)$ consisting of maps that send $\Sigma$ onto itself.  The set of connected components of $Aut(M, \Sigma)$ again forms a group, which we will denote $Mod(M, \Sigma)$.  

The 3-torus, $T^3 = S^1 \times S^1 \times S^1$ is known to have a unique (up to isotopy) genus three Heegaard splitting $(\Sigma, H_1, H_2)$~\cite{3torus}.  The mapping class group $Mod(T^3)$ is isomorphic to $SL(3, \mathbf{Z})$.  However, in $Mod(T^3, \Sigma)$ there are automorphisms that are non-trivial on $\Sigma$ but are isotopy trivial in $T^3$.  Thus $Mod(T^3,\Sigma)$ is in some sense much larger than $Mod(T^3)$.  We will prove the following:

\begin{Thm}
\label{mainthm}
For $(\Sigma, H_1, H_2)$ a (standard) genus three Heegaard splitting of $T^3$, $Mod(T^3, \Sigma)$ is generated by the automorphisms $\alpha_{12}$, $\alpha_{21}$, $\alpha_{13}$, $\alpha_{31}$, $\alpha_{23}$, $\alpha_{32}$, $\sigma$ and $\tau$ defined in Sections~\ref{mcgsect} and~\ref{gensect}.
\end{Thm}

This is the first non-trivial example of a genus three Heegaard splitting or of an irreducible Heegaard splitting for which a finite generating set can be explicitly described.  Goeritz~\cite{goer:s3} found a finite generating set for the mapping class group of the genus two Heegaard splitting of $S^3$.  Scharlemann~\cite{schar:mcg} recently published a new proof of this result after discovering that the two purported proofs of the higher genus cases are fatally flawed.  Akbas~\cite{akbas} later found a finite presentation for this group and shortly afterwards, Cho~\cite{cho} presented a new proof that this presentation is correct.

For genus one Heegaard splittings of lens spaces, the mapping class group is finite and easy to understand.  For minimal Heegaard splittings of connect sums of $S^1 \times S^2$, the mapping class group of the Heegaard splitting is isomorphic to the mapping class group of a handlebody, and is thus understood.  These two classes and the genus two Heegaard splitting of $S^3$ are the only previously understood examples.

The problem of understanding mapping class groups of Heegaard splittings is equivalent to problems in algebra and geometry:  Algebraically, the mapping class group is the intersection of the subgroups of $Mod(\Sigma)$ that extend into the two handlebodies.  These two subgroups are conjugates (by the gluing map) and each is finitely generated, but from the algebraic view point there appears to be no method for calculating their intersection in general.  

Geometrically, for genus greater than two, the mapping class group is the group of automorphisms of the curve complex preserving two handlebody sets.  Because the large scale geometry of the curve complex is understood, this point of view is useful for high distance Heegaard splittings.  (See~\cite{nam:mcg}.)  However, the local geometry of the curve complex is not well behaved, making this a difficult problem for low distance Heegaard splittings.  By appealing to the topology of $T^3$, we can solve this problem for this one case.  It seems that in general, solving this problem in the 3-manifold setting should be more reasonable than the equivalent problems in algebra and geometry.

\section{The mapping class group}
\label{mcgsect}

A \textit{Heegaard splitting} for a 3-manifold $M$ is a triple $(\Sigma, H_1, H_2)$ where $H_1, H_2 \subset M$ are handlebodies (connected manifolds homeomorphic to closed regular neighborhoods of graphs in $S^3$) and $\Sigma$ is a compact, connected, closed and orientable surface embedded in $M$ such that $H_1 \cup H_2 = M$ and $\partial H_1 = \Sigma = \partial H_2 = H_1 \cap H_2$.  

As defined above, $Mod(M, \Sigma)$ is the group of equivalence classes of orientation preserving automorphisms of $M$ that take $\Sigma$ onto itself.  Two automorphisms are equivalent if there is an isotopy from one to the other by automorphisms of $M$ that take $\Sigma$ onto itself.

Each connected component of $Aut(M, \Sigma)$ is a subset of a connected component of $Aut(M)$, so the inclusion map in $Aut(M)$ determines a canonical homomorphism $i : Mod(M, \Sigma) \rightarrow Mod(M)$.  In other words, $i$ is induced by ``forgetting'' $\Sigma$ and considering each element of $Mod(M, \Sigma)$ as an automorphism of $M$.  For further discussion of this homomorphism, see~\cite{finitemcg}.  The kernel of $i$ is the subgroup of $Mod(M,\Sigma)$ consisting of automorphisms that are isotopy trivial on $M$ but whose restrictions to $\Sigma$ are not isotopy trivial.

Consider $\mathbf{R}^3$ with axes labeled $x_1$,$x_2$,$x_3$.  Let $T_1$, $T_2$, $T_3$ be isometries of $\mathbf{R}^3$ where $T_i$ is translation by 1 unit along the axis $x_i$.  We can think of $T^3$ as the quotient of $\mathbf{R}^3$ by the group generated by $T_1$, $T_2$, $T_3$.  Each automorphism of $T^3$ lifts to an automorphism of $\mathbf{R}^3$.  Within the isotopy class for this automorphism, there is a representative that lift to $\mathbf{R}^3$ such that the automorphism of $\mathbf{R}^3$ fixes the origin.  It sends the vectors $(1,0,0)$, $(0,1,0)$, $(0,0,1)$ to integral vectors $v_1$, $v_2$, $v_3$, respectively.  Thus an automorphism of $T^3$ determines an element of the (integral) matrix group $GL_3(\mathbf{Z})$.

The matrix determined by an automorphism is unique and an element of $GL_3(\mathbf{Z})$ is represented by an automorphism of $T^3$ if and only if its determinant is one.  Thus $Mod(T^3)$ is isomorphic to the group $SL_3(\mathbf{Z})$ of three by three integral matrices with determinant one.  This group is generated by the six automorphisms that send $x_i$ to $x_i + x_j$ for $i \neq j$.  Let $A_{ij}$ be this automorphism for each distinct pair $i,j \in \{1,2,3\}$.

We would like to construct a Heegaard splitting for $T^3$ that fits naturally into the picture of the 3-torus described above.  Let $K_1 \subset T^3$ be the image in $T^3$ of the three edges in $\mathbf{R}^3$ from the origin to $(1,0,0)$, $(0,1,0)$ and $(0,0,1)$, respectively.  This $K_1$ is a graph with a single vertex and three edges.  Let $K_2$ be the image in $T^3$ of the same three edges, translated by the vector $(\frac{1}{2},\frac{1}{2},\frac{1}{2})$.  This is again a graph with one vertex and three edges.

Let $H_1$ be the set of points in $T^3$ whose distance to $K_1$ (in the euclidean metric on $T^3$) is less than or equaly to their distance to $K_2$.  Let $H_2$ be the set of points closer to $K_2$ and let $\Sigma$ be the set of points equidistant to $K_1$ and $K_2$.  Each of $H_1$, $H_2$ is the closure of a regular neighborhood of $K_1$, $K_2$, respectively so each set is a handlebody.  Moreover, $\Sigma$ is the boundary of each handlebody so $(\Sigma, H_1, H_2)$ is a (genus three) Heegaard splitting for $T^3$, shown in Figure~\ref{splittingfig}. (In the figure it is drawn as a smooth surface, though the way it's defined it is piecewise-linear.)   Boileau and Otal~\cite{3torus} showed that two Heegaard splittings of $T^3$ are isotopic if they have the same genus, so $(\Sigma, H_1, H_2)$ is the \textit{standard} (up to isotopy) genus three Heegaard splitting of $T^3$.
\begin{figure}[htb]
  \begin{center}
  \includegraphics[width=1.5in]{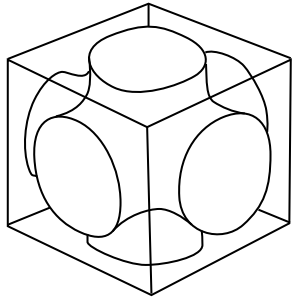}
  \caption{Gluing opposite faces of the cube by translations forms $T^3$.  The image in $T^3$ of the six-times punctured sphere shown here is the standard Heegaard splitting for $T^3$.}
  \label{splittingfig}
  \end{center}
\end{figure}

The parametrization of $\mathbf{R}^3$ induces an orientation on each edge of $K_1$, i.e. we have each edge point in the increasing direction along its axis.
Choose distinct $i,j,k \in \{1,2,3\}$.  Applying the transformation $A_{ij}$
to $T^3$ sends the edges of $K_1$ along the $x_j$ and $x_k$ axes onto themselves and sends the edge along the $x_i$ axis to the diagonal of a square in the $x_i$-$x_j$ plane.  Isotope this diagonal across the square into a neighborhood of the original spine so that it first passes along the edge in the $x_j$ axis and then along the edge in the $x_i$ axis.  

There is a unique element $\alpha_{ij}$ of $Mod(T^3,\Sigma)$ that maps to $A_{ij}$ in $Mod(T^3)$ and sends $K_1$ onto the new spine defined above. If we had chosen to slide the diagonal edge in the opposite direction across the square, we would have gotten a different element of $Mod(T^3,\Sigma)$.  Because each $\alpha_{ij}$ maps to $A_{ij}$, the images in $i$ of $\{\alpha_{ij}\}$ generate $Mod(T^3)$.  In order to extend this set of elements to a generating set for $Mod(T^3, \Sigma)$, we must understand the kernel of $i$.

\section{The kernel}
\label{gensect}

In this section, we will define elements $\sigma$ and $\tau$ of $Mod(T^3, \Sigma)$ that are non-trivial in $Mod(T^3,\Sigma)$, but isotopy trivial in $Mod(T^3)$.  First consider the translation of $\mathbf{R}^3$ by the vector $(\frac{1}{2}, \frac{1}{2}, \frac{1}{2})$.   This descends to an automorphism $\sigma$ of $T^3$ that sends the handlebody $H_1$ onto $H_2$ and sends $H_2$ onto $H_1$.  The translation of $\mathbf{R}^3$ is isotopic to the identity by a family of translations, inducing an isotopy of $T^3$ taking $\sigma$ to the identity.  Thus $\sigma$ is in the kernel of $i$.  Note that there are many automorphisms that are isotopic to the identity and swap the two handlebodies.  For our purposes, we could have picked any of these, but $\sigma$ happens to be the easiest to define.  (Note that $\sigma$ has order two in $Mod(T^3,\Sigma)$.)

The second automorphism in the kernel that we will define is called a torus twist and can be defined  on Heegaard splittings in a large family of manifolds.   Let $D_1 \subset H_1$ and $D_2 \subset H_2$ be properly embedded, essential disks such that the intersection of $\partial D_1$ and $\partial D_2$ is exactly two points.  We will assume that the algebraic intersection number of $\partial D_1$ and $\partial D_2$ is zero, though this is not necessary in general.  The author and Hyam Rubinstein~\cite{finitemcg} showed that $D_1$ and $D_2$ determine an element of the kernel of $i$ as follows:

Let $N$ be the closure of a regular neighborhood of $D_1 \cup D_2$.  Because there are two points of intersection, $N$ is a solid torus, as in Figure~\ref{localfig}.  (On the left, the disk $D_2$ is shown cut in half, at the top and bottom of the figure.)  Because the orientations at the two intersections are opposite, the surface $\Sigma \cap N$ is a four punctured sphere whose boundary consists of four simple closed curves in $\partial N$.   
\begin{figure}[htb]
  \begin{center}
  \includegraphics[width=3.5in]{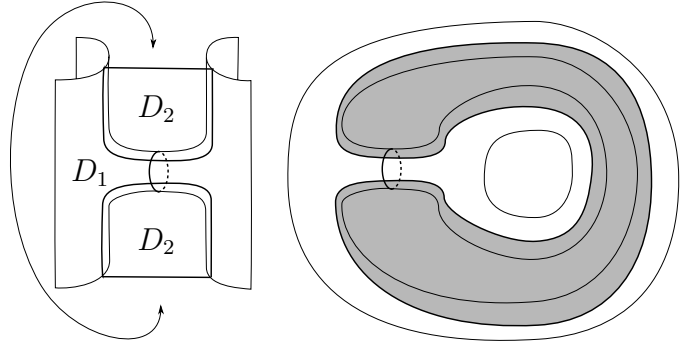}
  \put(-205,85){$D_2$}
  \put(-205,35){$D_2$}
  \put(-230,61){$D_1$}
  \caption{A neighborhood of the two disks is a solid torus.}
  \label{localfig}
  \end{center}
\end{figure}

The loops $\Sigma \cap \partial N$ are parallel longitudes in $\partial N$ so ``spinning'' $N$ along its longitude induces an automorphism of the solid torus $N$ that fixes the boundary of $N$ and takes $\Sigma \cap N$ onto itself.  This induces an automorphism of $\Sigma$ consisting of Dehn twists along loops parallel to $\Sigma \cap \partial N$.  We will call such an automorphism a \textit{torus twist}.  

Note that any conjugate of a torus twist is a torus twist.  We will choose a specific pair of disks and show that the kernel of $i$ is generated by torus twists that are conjugates of this fixed torus twist.

For each $i \in \{1,2,3\}$, the image in $T^3$ of the plane $\{(x_1,x_2,x_3) | x_i = 0\}$ intersects $H_2$ in a properly embedded, essential disk $D_2^i$ and the image of the plane $\{(x_1,x_2,x_3) | x_i = \frac{1}{2}\}$ intersects $H_1$ in a disk $D_1^i$.  For each $i$, $j$ the disks $D_1^i$ and $D_2^j$ are disjoint when $i = j$ and intersect in two points when $i \neq j$.  When $i \neq j$, a regular neighborhood of $D_1^i \cup D_2^j$ is a solid torus parallel to the $x_k$ axis ($k \neq i,j$) and we can twist in the positive direction along this axis.  Thus for each choice of distinct $i$ and $j$, $D_1^i$ and $D_2^j$ define a torus twist $\tau_{ij}$.  One of these is shown in Figure~\ref{twistorusfig}.
\begin{figure}[htb]
  \begin{center}
  \includegraphics[width=1.5in]{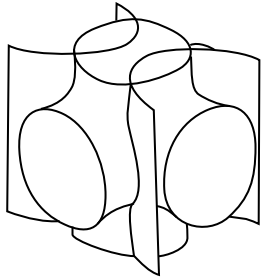}
  \caption{The torus defining an automorphism $\tau_{ij}$.}
  \label{twistorusfig}
  \end{center}
\end{figure}

For $i \neq j$, a regular neighborhood of $D_1^i \cup D_2^j$ intersects $\Sigma$ in loops parallel to the loops of intersection with a neighborhood of $D_1^j \cup D_2^i$.  However, twisting in the positive direction around the first soid torus induces oppositely oriented Dehn twists along these loops than twisting in the positive direction along the second solid torus.  Thus $\tau_{ij} = \tau_{ji}^{-1}$.  For our last generator, define $\tau = \tau_{12}$.

\section{Torus Twists}
\label{tortwistsect}

Let $G \subset Mod(M,\Sigma)$ be the subgroup generated by $\{\alpha_{ij}, \sigma, \tau\}$.  In order to show that $G = Mod(M,\Sigma)$, we will show first that the kernel of $i$ is generated by a certain class of torus twists and second that these torus twists are all conjugate to $\tau$ by elements of $G$.  We will begin by defining some relations on elements of $G$.  Define $r_{ij} = \alpha_{ij} \alpha^{-1}_{ji} \alpha_{ij} \tau_{jk}$.  The reader can check the following:

\begin{Lem}
Let $i,j,k \in \{1,2,3\}$ be distinct.  Then the automorphism $r_{ij}$ takes the image in $T^3$ of the $i$-axis onto the image of the $j$ axis and takes the $j$ axis onto the negative $i$ axis.
\end{Lem}

Because $\tau = \tau_{12} = \tau^{-1}_{21}$, the rotations $r_{31} = r^{-1}_{13}$ and $r_{32} = r^{-1}_{31}$ are in $G$.  Moreover, we have that $\tau_{13} = r_{23} \tau_{12} r^{-1}_{23}$ so $\tau_{13}$ is in $G$, as is $r_{12}$ and therefore $\tau_{23}$.  Thus the torus twists $\{t_{ij}\}$ are all contained in $G$.

Let $S$ be the image in $T^3$ of the plane $\{x_1,x_2,x_3 | x_3 = 0\}$ in $\mathbf{R}^3$.  Let $\phi \in Mod(M,\Sigma)$ be a torus twist defined by disks $D_1 \subset H_1$ and $D_2 \subset H_2$.  Let $T$ be th boundary of a regular neighborhood of $D_1 \cup D_2$, i.e. the torus that defines the twist.  Assume $S$ and $T$ are transverse.

\begin{Lem}
\label{conjtwistlem}
If $T \cap S$ is connected and essential in $T$ and $\Sigma \setminus T$ is planar then $\phi$ is conjugate to $\tau$ by elements of $G$ (and thus $\phi$ is in $G$).
\end{Lem}

\begin{proof}
Because the intersection $T \cap S$ is a single loop and this loop is essential in $T$, the complement $T \setminus S$ is an annulus $A \subset T$.  The intersection of $T$ with $\Sigma$ consists of four loops, each of which intersects $A$ in a properly embedded, essential arc.  Because $\Sigma \setminus T$ is planar, these arcs cut the twice punctured, genus two surface $\Sigma \setminus S$ into planar pieces.  Because the pieces are planar, the arcs must cut $\Sigma \setminus T$ into two annuli.  

The intersection $H_1 \cap S$ is a punctured torus $S' \subset S$.  The intersection of $T$ with $S'$ is a pair of properly embedded arcs (with endpoints in $T \cap \Sigma$) that cut $S'$ into a disk and an annulus.  There is a spine for $S'$ such that one edge of the spine intersects $T$ in two points and the other edge is disjoint from $T$.  Let $\gamma_1$ be the edge that intersects $T$ and $\gamma_2$ the edge disjoint from $T$.

Let $\gamma'_1$ be the edge of the spine $K_1$ for $H_1$ that is the image of the $x_1$ axis of $\mathbf{R}^3$.  Let $\gamma'_2$ be the edge of $K$ coming from the $x_2$ axis.  These two edges form a spine for $S$.  By applying the generators $\alpha_{12}$, $\alpha_{21}$, $\tau_{13}$ and $\tau_{23}$, one can send $\gamma'_1$, $\gamma'_2$ onto any spine for $S'$.  In particular, one can send these edges onto $\gamma_1$ and $\gamma_2$.  This automorphism sends the torus defined by $D_1^1$ and $D_2^2$ onto $T$ so conjugating $\tau$ by this automorphism produces $\phi$.
\end{proof}

\section{Defining disks}
\label{disksect}

Let $D_1 \subset H_1$ and $D_2 \subset H_2$ be properly embedded, essential disks in the standard Heegaard splitting $(\Sigma, H_1, H_2)$ of $T^3$.

\begin{Def}
The disks $D_1$, $D_2$ are a \textit{defining pair} if the boundary of $D_1$ can be isotoped disjoint from $\partial D_2$.
\end{Def}

Each of the pairs of disks $D_1^i$, $D_2^i$ constructed above is a defining pair.

\begin{Lem}
\label{definetoruslem}
A defining pair of disks determines a unique isotopy class of incompressible tori in $T^3$.
\end{Lem}

\begin{proof}
First note that $(\Sigma, H_1, H_2)$ is irreducible because cutting $\Sigma$ along a reducing sphere would produce a genus two Heegaard splitting.  This is impossible because $\pi_1(T^3)$ has rank three.  Thus the boundaries of $D_1$ and $D_2$ cannot be isotopic.  

We will see momentarily that both $D_1$ and $D_2$ must be non-separating disks in $H_1$, $H_2$, respectively, but for now note that if $D_1$ is separating, then it cuts $\Sigma$ into a genus two handlebody and a  genus one handlebody.  The boundary of $D_2$ must be contained in the boundary of the genus two handlebody so a (non-separating) meridian disk for the genus one handlebody will be disjoint from $D_2$.  Thus if $D_1$ is separating then it uniquely determines a non-separating disk in $H_1$ disjoint from $\partial D_2$ and we can replace $D_1$ with this disk.  Likewise, if $D_2$ is separating then we can replace $D_2$ with the unique non-separating disk in $H_2$ determined by $D_2$.  Thus we can assume $D_1$ and $D_2$ are non-separating.

Casson and Gordon~\cite{cgred} showed that given a pair $D_1 \subset H_1$, $D_2 \subset H_2$ of disjoint essential disks, compressing $\Sigma$ across $D_1$ and $D_2$ produces a separating surface $S$ that is incompressible or compresses to either an incompressible surface or to a reducing sphere for the Heegaard splitting.  Because the genus three Heegaard splitting of $T^3$ is irreducible, $S$ is incompressible or compresses to an incompressible surface.

Because $\partial D_1$ and $\partial D_2$ are non-separating in the genus three surface $\Sigma$, the surface $S$ consists of one or two tori.  The only incompressible surfaces in $T^3$ are non-separating tori, so $S$ must consist of two parallel incompressible tori whose union is separating, though each is non-separating on its own.  There was no choice involved in the construction of $S$, and the two components of $S$ are isotopic so $D_1$ and $D_2$ define a unique isotopy class of incompressible tori.  The tori defined by $D_1^1$ and $D_2^1$ are shown in Figure~\ref{definingfig}.
\end{proof}
\begin{figure}[htb]
  \begin{center}
  \includegraphics[width=3.5in]{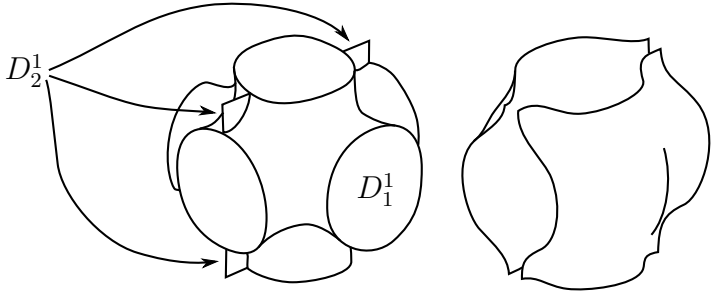}
  \put(-268,90){$D_2^1$}
  \put(-135,45){$D_1^1$}
  \caption{Compressing $\Sigma$ along a defining pair produces a pair of incompressible tori.}
  \label{definingfig}
  \end{center}
\end{figure}

\begin{Lem}
\label{uniquepairlem}
If $D_1$ and $D_2$ are a defining pair then each of $D_1$ and $D_2$ is non-separating in $\Sigma$.  If $D'_2$ is an essential, properly embedded disk in $H_2$ disjoint from $\partial D_1$ then $D'_2$ is isotopic to $D_2$.  Similarly, if $D'_1$ is essential and properly embedded in $H_1$ and $\partial D'_1$ is disjoint from $\partial D_2$ then $D'_1$ is isotopic to $D_1$.
\end{Lem}

In other words, a disk in $H_1$ or $H_2$ is half of at most one defining pair.  (Some disks are not half of any defining pair.)

\begin{proof}
Let $D_1$ and $D_2$ be a defining pair for $(\Sigma, H_1, H_2)$.  First assume that each of $D_1$ and $D_2$ is non-separating.  As in the proof of Lemma~\ref{uniquepairlem}, compressing $\Sigma$ along $D_1$ and $D_2$ produces a surface $S$ consisting of two parallel incompressible tori.  If we first compress $S$ across $D_1$ then we get a surface isotopic to the boundary of $H_2 \cup N(D_1)$, where $N(D_1)$ is the closure of a regular neighborhood of $D_1$.  

Compressing $\partial(H_2 \cup N(D_1))$ across $D_2$ produces a pair of parallel tori that cut $T^3$ into two pieces homeomorphic to $T^2 \times [0,1]$.  Thus $H_2 \cup N(D_1)$ is homeomorphic to attaching a one-handle to $T^2 \times [0,1]$ with ends on opposite boundary tori.  Any boundary compressing disk for $H_2$ that is disjoint from $D_1$ will be a boundary compressing disk for $H_2 \cup N(D_1)$.  Any boundary compressing disk for $H_2 \cup N(D_1)$ is isotopic to a meridian of the one handle so there is a unique compressing disk for $H_2$ disjoint from $D_1$.  In other words and compressing disk $D'_2$ will be isotopic to $D_2$.  The same argument with $D_1$ and $D_2$ reversed shows that any disk in $H_1$ disjoint from $D_2$ is isotopic to $D_1$.

If $D_2$ is separating and $D_1$ is non-separating then as noted in the proof of Lemma~\ref{definetoruslem} there is a non-separating disk $D'_2$ that is disjoint from both $D_1$ and $D_2$.  Because $D_1$ and $D'_2$ are disjoint and each is non-separating, we have just shown that $D_2$ must be isotopic to $D'_2$.  Thus if $D_1$ is non-separating then $D_2$ must be non-separating.

Finally, if $D_1$ is separating then there is a non-separating disk $D'_1$ disjoint from $D_1$ and $D_2$.  Because $D'_1$ and $D_2$ are disjoint and $D'_1$ is non-separating, $D_2$ must be non-separating.  Because $D'_1$ and $D_2$ are disjoint and each is non-separating, $D_1$ must be isotopic to $D'_1$, so $D_1$ is non-separating.  This proves the Lemma in its full generality.
\end{proof}

\section{Graphics and disks}
\label{grndisksect}

The key to the proof of Theorem~\ref{mainthm} is the following Lemma:

\begin{Lem}
\label{twistslem1}
Let $D_1$, $D_2$ and $D'_1$,$D'_2$ be disjoint pairs.  Then the incompressible tori determined by the two pairs are isotopic in $T^3$ if and only if there is a sequence of torus twists on $(\Sigma, H_1, H_2)$ taking $D'_1$, $D'_2$ onto $D_1$, $D_2$ such that each torus twist is along a torus that cuts $\Sigma$ into planar surfaces and intersects the incompressible torus defined by $D_1$, $D_2$ in a single loop.
\end{Lem}

This will be proved in Section~\ref{isotopysect}.  The proof relies on an understanding of Morse functions and stable functions on 3-manifolds, which we will describe in this and the next sections.

Let $f : T^3 \rightarrow \mathbf{R}$ be a Morse function on $T^3$ with one index zero  critical point at level 0, three index one then three index two critical points (at distinct levels), then one index three critical point at level 1.  

Let $\pi : T^3 \rightarrow S^1$ be a fiber bundle map such that $\pi^{-1}(s)$ is a torus for each $s \in S^1$.  We can think of $\pi$ as a circle valued Morse function on $T^3$.  As was shown in~\cite{stabs}, after an arbitrarily small isotopy the product of two Morse functions on a 3-manifold $M$ is a stable function from $M$ to $\mathbf{R}^2$.  The same argument in this situation implies that after an arbitrarily small isotopy, the product of $f$ and $\pi$ is a stable function $f \times \pi : T^3 \rightarrow [0,1] \times S^1$.

Assume $f$ and $g$ have been isotoped so that $F = f \times \pi$ is stable.  For each $t \in [0,1]$, define $f_t$ to be the level surface $f^{-1}(t)$.  Define $\pi_t$ to be the restriction $\pi|_{f_t}$.  As was shown in~\cite{stabs}, the discriminant set $\mathcal{J}$ of $F$ is a one dimensional submanifold in $T^3$, whose image $F(\mathcal{J})$ is a finite graph in $[0,1] \times S^1$.  We will say that $F$ is \textit{generic} if $F$ is stable and for any $s \in S^1$, there is at most one vertex of $F(\mathcal{J})$ in the line $[0,1] \times \{s\}$.  

Because $f$ is Morse, $f_t$ is a surface for all but finitely many values of $t$.  Because $F$ is stable, $\pi_t$ is a Morse function on $f_t$ for all but finitely many of the regular values of $t$.  If $F$ is generic then at the values of $t$ where $\pi_t$ fails to be Morse, $\pi_t$ will have either exactly two critical points at the same level or exactly one degenerate critical point.  (If $F$ is stable but not generic, there will be a value of $t$ where there are more than one pairs of critical points at the same level.)  Let $\ell_1 \in [0,1]$ be the level of the highest index one critical point in $f$ and let $\ell_2$ be the level of the lowest index two critical points.  (By assumption, $\ell_1 < \ell_2$.)

Assume some level set $f_t$ is equal to $\Sigma$ (as a subset of $T^3$) for some $t \in (\ell_1,\ell_2)$.  Then the surfaces $\{f_t | t \in (\ell_1,\ell_2)\}$ determine an isotopy of $\Sigma$, inducing a canonical (up to isotopy) identification $c_t : \Sigma \rightarrow f_t$ for each $t \in (\ell_1,\ell_2)$.

\begin{Lem}
\label{grapchicdefinesdiskslem}
Assume $f$ has the property that for any regular value $t < \ell_1$, $f_t$ is not compressible into $f^{-1}([t,1])$ and for any regular $t > \ell_2$, $f_t$ is not compressible into $f^{-1}([0,t])$.  If $f \times \pi : T^3 \rightarrow [0,1] \times S^1$ is generic then there is a unique (up to isotopy) disjoint pair $D_1$, $D_2$ such that for some $t \in (\ell_1, \ell_2)$, $c_t$ sends $\partial D_1$ and $\partial D_2$ to loops isotopic into regular level sets of $\pi_t$.
\end{Lem}

We should emphasize that there are essentially two parts to the statement: first that for some $t$, the level sets of $\pi_t$ determine a disjoint pair of disks and second that if $\pi_t$ and $\pi_{t'}$ determine disjoint pairs, then they determine the same disjoint pair.

\begin{proof}
Let $C_1 \subset [0,1]$ be the set of points $t \in [0,1]$ such that an essential level loop of $\pi_t$ bounds a disk in $f^{-1}([0,t])$.  Similarly, let $C_2$ be the set of points where an essential level loop of $\pi_t$ bounds a disk in $f^{-1}([t,1])$.  We first will show that $C_1 \cap C_2$ is a non-empty, connected, open interval $(a,b) \subset [0,1]$, implying that the level loops of $\pi_t$ for $t \in (a,b)$ contain the boundaries of a defining pair.

For any $\varepsilon > 0$, the set $f([\ell_1 + \varepsilon, \ell_2 - \varepsilon]) \subset M$ is foliated by level surfaces of $f$ and its complement is a pair of handlebodies.  Thus there is a sweep-out $f' : M \rightarrow [0,1]$ that agrees with $f$ on the interval $[\ell_1 + \varepsilon, \ell_2 - \varepsilon]$.  (Each level set of $f'$ is isotopic to $\Sigma$.)  Define the sets $C'_1, C'_2 \subset [0,1]$ as the values of $t$ for which the level set of $f'$ intersects a level set of $\pi$ in a loop bounding a disk to one side of $t$ or the other.

Bachman and Schleimer~\cite[Claim 6.7]{bachscl} showed that for a sweep-out $f'$ of a surface bundle, the set $C'_1$ is of the form $[0,b)$ for some $b$ and $C'_2$ is of the form $(a,0]$ for some $a$.  Thus if $a < b$ then $C'_1 \cap C'_2$ is an open interval $(a,b)$.  Because $f$ agrees with $f'$ on $[\ell_1 + \varepsilon, \ell_2 - \varepsilon]$ for every $\varepsilon > 0$, the intersection of $C_1 \cap C_2$ with $(\ell_1,\ell_2)$ is a (possibly empty) open interval.

Note that if for some $t$ and $s$ there are regular loops of $\pi_t(s)$ that are essential in $f_t$ and trivial in $\pi^{-1}(s)$ then $t$ is in either $C_1$ or $C_2$.  Cooper and Scharlemann~\cite{sch:solv} showed that for a torus bundle $M$ with bundle map $\pi$, if there is a surface $S \subset M$ such that $\pi|_S$ is Morse and every regular level loop is either trivial in both surfaces or essential in both surfaces then $S$ is an essential torus.  If $\pi|_S$ is near Morse (i.e. at a crossing of the graphic) and all regular level loops are trivial or essential in both surfaces then $S$ is a strongly irreducible, genus two Heegaard surface.  

Because $\Sigma$ is not a genus two surface, every level surface $f'_t$ must have a level loop that is essential in $f'_t$ but trivial in the appropriate level surface of $\pi$.  Thus $C'_1 \cap C'_2$ is a non-empty open interval.  The set $C_1 \cap C_2$ is the intersection of $C'_1 \cap C'_2$ with 
 $[\ell_1,\ell_2]$ (because $f$ and $f'$ agree on this set) so $C_1 \cap C_2$ is empty if and only if $C_1 \cap (\ell_1,\ell_2)$ or $C_2 \cap (\ell_1,\ell_2)$ is empty.

Assume for contradiction $C_1$ is disjoint from $(\ell_1,\ell_2)$.  Let $t$ be a value between $\ell_1$ and the last crossing before $\ell_1$.  The surface $f_t$ has genus two and $\pi|_{f_t}$ is Morse so there is a regular level $s \in S^1$ such that a loop in $f_t^{-1}(s)$ is essential in $f_t$ and trivial in $\pi^{-1}(s)$.  Thus some essential loop $\gamma \subset f_t^{-1}(s)$ bounds a disk in either $f^{-1}([0,t])$ or $f^{-1}([t,1])$.  

If $\gamma$ bounds a disk in $f^{-1}([0,t])$ then attaching a one-handle at the critical point at time $\ell_1$ does not affect this disk.  Thus $\gamma$ must bound a disk in $f^{-1}([t,1])$.  This contradicts the assumption that for any regular value $t < \ell_1$, $f_t$ is not compressible into $f^{-1}([t,1])$.  Thus $C_1$ must intersect $(\ell_1,\ell_2)$.  A similar argument implies that $C_2$ must intersect $(\ell_1,\ell_2)$.  Thus $C_1 \cap C_2$ is a non-empty, connected, open interval $(a,b)$.

For $t \in (a,b)$, at least one loop in the pair of pants decomposition bounds a disk in $H_1$ and at least one bounds a disk in $H_2$.  Because any two loops in the pants decomposition are disjoint, Lemma~\ref{uniquepairlem} implies that there is exactly one loop bounding a disk in $H_1$ and exactly one bounding a disk in $H_2$.  

At each crossing in $(a,b) \times S^1$, $f_t$ passes through a near-Morse function and the induced pants decomposition changes in one of two ways (See~\cite{HT:rep}): (1) One loop in the pants decomposition may be replaced by a new loop that intersects the original loop in one or two points or (2) two loops that are contained in a twice punctured torus component are simultaneously removed and replaced.  After the change, there are still loops bounding disks in opposite handlebodies.  The second type of move cannot change the two loops that bound a defining pair because two such loops are never contained in a twice punctured torus component.  Because the loops bounding the disks cannot both change simultaneously, Lemma~\ref{uniquepairlem} implies that neither can change.  Thus the defining pair is uniquely determined.
\end{proof}

\section{Stable functions in dimension three}
\label{dim3sect}

We have seen how the graphic defined by a Morse function $f$ for a genus three Heegaard splitting for $T^3$ and a torus bundle map $\pi$ for $T^3$ determine a defining pair of disks for the Heegaard splitting.  An automorphism $\phi$ of the Heegaard splitting will take this defining pair to a new defining pair, determined by the graphic of the original sweep-out $f$ and a new bundle map $\pi \circ \phi$.  If $\phi$ is isotopy trivial then the isotopy determines a family of bundle maps $\{\pi^t\}_{t \in [0,1]}$ such that $\pi^0 = \pi$ and $\pi^1 = \pi \circ \phi$.

Each bundle map $\pi^t$ determines a graphic with $f$.  At each value of $t$ where the graphic is generic, $\pi^t$ determines a pair of defining disks.  Thus the family $\{\pi^t\}$ determines a sequence of defining disks for $(\Sigma,H_1, H_2)$.  If we choose this family carefully, we can understand the sequence of defining disks well enough to find a sequence of automorphisms of $(\Sigma, H_1, H_2)$ taking the original defining pair to each consecutive pair in the sequence.  In this section we will describe how the graphic can change during the isotopy of $\pi$ and in the next section we will show how this corresponds to a sequence of defining disks which suggest a sequence of automorphisms of $(\Sigma, H_1, H_2)$.

In order to understand how the graphic changes, we will consider an isotopy of $f$ rather than an isotopy of $\pi$.  Because $\phi$ is isotopic to the identity, there is a continuous family $\{\phi_t : T^3 \rightarrow T^3\}$ such that  $\phi_0$ is the identity and $\phi_1 = \phi$.  The family of bundle maps defined above is given by $\pi^t = \pi \circ \phi_t$ and the graphic at time $t$ is determined by $f \times (\pi \circ \phi_t)$.  If we compose the stable function with $\phi^{-1}_t$, we find that the graphic is also given by the map $(f \circ \phi^{-1}_t) \times \pi$.

Define $f^t = f \circ \phi^{-1}_t$.  Because $f^t$ is a Morse function for each $t$, there is an open neighborhood $N_t \subset C^\infty(T^3,\mathbf{R})$ (with the Whitney $C^\infty$ topology~\cite{golub}) such that each function in $N$ is isotopic to $f^t$.  Moreover, $N_t$ can be chosen to be convex in $C^\infty(T^3,\mathbf{R})$.  Because the set $\{f^t\} \subset C^\infty(T^3,\mathbf{R})$ is compact, it is covered by a finite subset of the convex open neighborhoods $\{N_t\}$.  Because each neighborhood is convex and consists of Morse functions isotopic to $f$, the path $f^t$ can be replaced with a piecewise-linear path consisting of arcs connecting consecutive functions $g^0,\dots,g^n \in \{f^t\}$ such that each $g^i \in N_i \cap N_{i+1}$ determines a generic graphic with $\pi$ and each arc determines an isotopy of $g^i$ onto $g^{i+1}$.  Thus in order to understand how the graphic changes as $f^t$ changes, we can restrict our attention to straight arcs in $C^\infty(T^3,\mathbf{R})$.

Each intermediate function in the arc from $g^i$ to $g^{i+1}$ is of the form $a g^i + b g^{i+1}$ where $a,b > 0$ and $a + b = 1$.  By scaling this function as in~\cite{stabs}, we can make it of the form $\cos(s)g^i + \sin(s)g^{i+1}$.  Thus we are interested in the graphic defined by the stable function $(\cos(s)g^i + \sin(s)g^{i+1}) \times \pi$.  This is the projection of the map $g^i \times g^{i+1} \times \pi : T^3 \rightarrow \mathbf{R}^1 \times S^1$ onto the annulus $L \times S^1$  where $L \in \mathbf{R}^2$ is the line through the origin with slope $\frac{\cos(s)}{\sin(s)}$.  Thus we can understand the effect of the isotopy on the graphic by looking at projections of the map $g^i \times g^{i+1} \times \pi$.

We can choose the sequence $\{g^i\}$ such that each of these maps $g^i$, $g^{i+1}$, $\pi$, $g^i \times g^{i+1}$, $g^i \times \pi$ and $g^{i+1} \times \pi$ is stable, so each is contained in an open ball of isotopic maps in its respective vector space.  The projections of $\mathbf{R}^2 \times S^1$ into the appropriate subspaces define continuous maps from $C^\infty(T^3, \mathbf{R}^3 \times S^1)$ into the spaces containing the above maps and the preimage of each open neighborhood is open in $C^\infty(T^3, \mathbf{R}^2 \times S^1)$.  Their intersection is an open set containing $g^i \times g^{i+1} \times \pi$.  

Mather~\cite{mather} showed that stable functions between three dimensional manifolds are dense in the Whitney $C^\infty$ topology, so this open neighborhood contains a stable function from $T^3$ into $\mathbf{R}^2 \times S^1$.  This stable function projects to maps isotopic to $g^i$, $g^{i+1}$, $\pi$, $g^i \times g^{i+1}$, $g^i \times \pi$ and $g^{i+1} \times \pi$ so if we replace $g^i$, $g^{i+1}$ and $\pi$ with these isotopic maps, $g^i \times g^{i+1} \times \pi$ will be stable and the rest of the maps will be isotopic to the original maps.

Mather's classification~\cite{mather} of singularities of stable functions between three dimensional manifolds implies that at each point $p$ in $T^3$, some neighborhood $N$ of $p$ can be parameterized and some open ball in $\mathbf{R}^2 \times S^1$ can be parameterized such that $f|_N$ has one of the following forms:

$$f(x,y,z) = (x,y,z),$$
$$f(x,y,z) = (x^2,y,z),$$
$$f(x,y,z) = (xy+x^3,y,z)$$
or
$$f(x,y,z) = (xy + x^2z + x^4,y,z).$$

The first type of point on the list is a regular point and at such a point, the discriminant map from $T(T^3)$ to $T(\mathbf{R}^2 \times S^1)$ is one-to-one.  At the last three types of points, the discriminant map has a one dimensional kernel, so these points are in the discriminant set.  The discriminant set intersects each such neighborhood in an open disk, so the discriminant set of $g_0 \times g_1 \times \pi$ is a compact 2-dimensional submanifold in $T^3$.  

The image of the discriminant set in $\mathbf{R}^2 \times S^1$ is an immersed 2-manifold with ``cusps''.  The cusps in the immersion of the 2-manifold form edges consisting of points with neighborhoods of the third type, and these edges come together at points with neighborhoods of the last type.

\begin{Lem}
\label{graphichangeslem}
Given a Morse function $f : T^3 \rightarrow \mathbf{R}$, and isotopic torus bundle maps $\pi^0$, $\pi^1$ such that $f \times \pi^0$ and $f \times \pi^1$ are stable then the maps $f \times \pi^0$ and $f \times \pi^1$ are related by a sequence of isotopies and moves of the types shown in Figure~\ref{gmovesfig}.
\end{Lem}

In fact, the theorem is true for any generic path of maps from a 3-manifold to a 2-manifold, but we do not need to prove it in such generality, so we will stick to the situation described above.
\begin{figure}[htb]
  \begin{center}
  \includegraphics[width=3.5in]{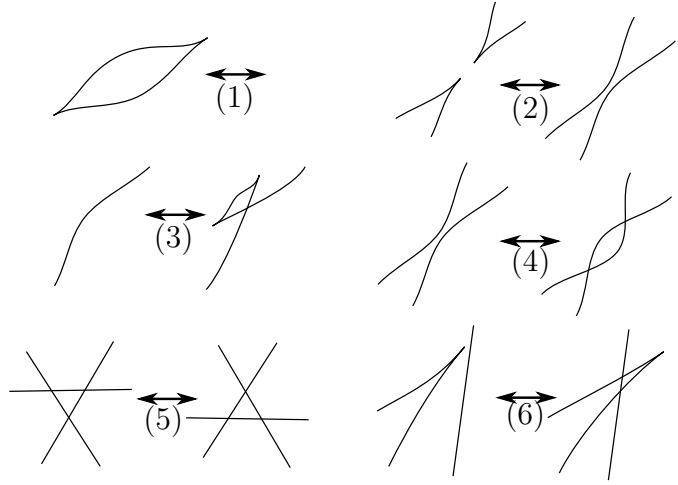}
  \put(-175,142){(1)}
  \put(-63,138){(2)}
  \put(-198,89){(3)}
  \put(-63,80){(4)}
  \put(-202,20){(5)}
  \put(-65,22){(6)}
  \caption{Any two graphics are related by isotopies and a sequence of these six moves.}
  \label{gmovesfig}
  \end{center}
\end{figure}

\begin{proof}
Because $\pi^0$ and $\pi^1$ are isotopic, there is a smooth family of stable functions $\{\pi^t | t \in [0,1]\}$ from $\pi^0$ to $\pi^1$.  As noted above, we can approximate the family of stable functions $f \times \pi^t = f \times (\pi \circ \phi_t)$ by considering the isomorphic stable functions $(f \circ \phi_t^{-1}) \times \pi$, then approximating the path $\{f \times \phi_t^{-1}\}$ by a piecewise-linear path of straight arcs between a sequence of functions $\{g^i \in C^\infty(T^3,\mathbf{R})\}$.  The stable function defined by $\pi$ and a function in the arc from $g^i$ to $g^{i+1}$ can be recovered as a projection of $g^i \times g^{i+1} \times \pi$ onto $L \times S^1$ for some line $L$ in $\mathbf{R}^2$.  Thus in order to understand how the graphic changes with $t$, we must consider projections of $g^i \times g^{i+1} \times \pi$.  Without loss of generality, we will consider projections of $g^0 \times g^1 \times \pi$.

Let $S \subset T^3$ be the discriminant set of the map $F = g^0 \times g^1 \times \pi : T^3 \rightarrow \mathbf{R}^2 \times S^1$.  As noted above, $S$ is a compact, closed surface in $T^3$.  The image of $S$ in $\mathbf{R}^2 \times S^1$ is surface with ``cusps''.  On the complement of the cusps, there is a well defined map from $T_p S$ to $T_{F(p)} (\mathbf{R}^2 \times S^1)$ where $p$ is a non-cusp point in $S$.  If the plane tangent to $F(S)$ at $F(p)$ is not parallel to the plane $\mathbf{R}^2 \times \{\pi(p)\}$ then its intersection with this plane determines a slope in $\mathbf{R}^2$.  For each $p \in S$, let $s(p)$ be this slope.  The function $s$ is defined on the complement in $S$ of the points where $F(S)$ is parallel to the plane $\mathbf{R}^2 \times \{\pi(p)\}$.

By perturbing $F(S)$ slightly (by isotoping $g^0$, $g^1$ and $\pi$), we can ensure that the function $s$ on $S$ is a Morse function defined on the complement in $S$ of a finite number of points.  Locally, we can identify  a patch of $F(S)$ with the graph of a function $\gamma$ from $\mathbf{R}^2$ (in variables $y$, $z$) to $\mathbf{R}$ as follows:  Let $\gamma' : \mathbf{R} \rightarrow \mathbf{R}$ be a smooth function.  For each $a,b$, define $\gamma(a,b) = \gamma'(a) + \int_0^b s\ dy$, where the integral is taken along the arc from $(a,0)$ to $(a,b)$.  The slope of the intersection with the plane  $\mathbf{R}^2 \times \{\pi(p)\}$ is precisely $\frac{d\gamma}{dy}$ so for some $\gamma'$, $F(S)$ is the graph of $\gamma$.

Let $p_t$ be orthogonal projection of $\mathbf{R}^2 \times S^1$ onto the annulus $L \times S^1$ where $L \subset \mathbf{R}^2$ is the line through the origin with slope $\cos(t) / \sin(t)$.  The composition of $F$ with $p_t$ is a function from $M$ to $\mathbf{R}^2$ and the discriminant set of $F \circ p_t$ is the projection of the closure in $S$ of the subset $s^{-1}(t)$.  In order to understand how the graphic changes with $t$, we must understand how $s^{-1}(t)$ maps to the graphic.

Because $s$ is a Morse function, $s^{-1}(t)$ will be a collection of closed loops in $S$.  By recovering $S$ from $s$ as the graph of a function $\gamma$ (defined up to choice of $\gamma'$), one can check that the graphic defined by $F \circ p_t$ is related to the level sets as follows:  

If the image of $s^{-1}(t)$ in $F(S)$ is transverse to the plane $\mathbf{R}^2 \times \{\pi(p)\}$ at a point $p \in S$ then in the projection, $p$ maps into the interior of an edge in the graphic.  If the level set is tangent to the plane at $F(p)$ then $p$ maps to a cusp in the graphic.  Each cusp point of $S$ in the closure of $s^{-1}(t)$ maps to a cusp in the graphic.  A non-cusp point $p$ where $s$ is not defined (because $F(S)$ is tangent to the plane  $\mathbf{R}^2 \times \{\pi(p)\}$) maps to the interior of an edge.  (Such a point is in the closure of $s^{-1}(t)$ for every $t$.)

When $t$ passes through a critical level of $s$, there are two cases to consider: If the critical points has index zero (or two) then right after (right before) the critical point, a component of  $s^{-1}(t)$ will be tangent to the plane $\mathbf{R}^2 \times \{\pi(p)\}$ in exactly two points.  Its image in $F \circ p_t$ will be an eye as in move (1).  Before (after) the critical point, there is no such component so when $t$ passes through the critical level, the eye is created (removed) as in move (1).  Similar reasoning shows that the graphic changes by move (2) if the critical point has index one.  

When $t$ passes through a level where tangencies to the plane $\mathbf{R}^2 \times \pi(p)$ are created or removed, the tangencies are created or removed in pairs, and in the graphic this corresponds to move (3).  At the points where $F(S)$ is tangent to the plane, the graphic does not change at all.  

Generically, there will be no critical point of $s$ in the (one dimensional) set of cusp points in $S$.  Thus the graphic only changes along the set of cusps when a level set of $s$ becomes tangent to a arc of cusp points.  In this case, two cusps in the graphic are created of eliminated as in move (3).

A stable map from a 3-manifold to $\mathbf{R}^2$ can fail to be stable for two reasons: there may be critical points that do not have neighborhoods of the necessary forms, or the image of the disciminant set may be non-generic, i.e. have triple points, tangencies or double points at cusps. We have shown that there is a family of functions from $f \times \pi_0$ to $f \times \pi_1$ such that there are finitely many intermediary functions with unstable neighborhoods, and these correspond to the first three moves.  In between these functions, the graphic changes by some homotopy of the image of the discriminant set.  By perturbing the family $\{f_t\}$ slightly we can ensure that the homotopy is generic, consisting of the a finite sequence of moves (4), (5) and (6).  This completes the proof.
\end{proof}

\section{Graphics and isotopies}
\label{isotopysect}

\begin{proof}[Proof of Lemma~\ref{twistslem1}]
A torus twist determines an automorphism that is isotopy trivial in $T^3$ so it takes the torus determined by a defining pair of disks to an isotopic torus.  If two defining pairs are related by a sequence of torus twists then by induction the induced tori are isotopic.  The majority of the proof of Lemma~\ref{twistslem1} will be devoted to the converse of this.

Let $f : T^3 \rightarrow [0,1]$ be a Morse function with one index zero critical point, three index one critical points followed by three index two critical points, then one index three critical point.  Let $\ell_1$ be the level of the last index one critical point and $\ell_2$ the level of the first index three critical point.  Any level set $f_t = f^{-1}(t)$ for $t \in (\ell_1,\ell_2)$ is a Heegaard surface for $T^3$ and is thus isotopic to $\Sigma$ for the standard (and unique) Heegaard splitting $(\Sigma, H_1, H_2)$.  Moreover, there is a continuous family $\{c_t : \Sigma \rightarrow f_t : t \in (\ell_1,\ell_2)\}$ identifying each level surface with $\Sigma$.

Assume $f$ has the property that for $t < \ell_1$, $f_t$ is not compressible into $f^{-1}([t,1])$.  This will be the case whenever the meridian disk defined by the last index one critical point is not part of a defining pair, so such an $f$ exists.  Similarly, assume that for $t > \ell_2$, $f_t$ is not compressible into $f^{-1}([0,t])$.

For a defining pair $D_1$, $D_2$, there is a bundle map $\pi^0 : T^3 \rightarrow S^1$ such that the defining pair determined by $f \times \pi^0$ and the maps $c_t$ is isotopic to $D_1$ and $D_2$.  Likewise, there is a bundle map $\pi^1$ such that $f \times \pi^1$ determines a defining pair isotopic to the second pair $D'_1$, $D'_2$.  By assumption the defining pairs define isotopic incompressible tori, so $\pi^0$ and $\pi^1$ are isotopic.  Thus Lemma~\ref{graphichangeslem} implies that $f \times \pi^0$ and $f \times \pi^1$ are related by a sequence of the moves shown in Figure~\ref{gmovesfig} and isotopies of the graphic.

Because we chose $f$ so that $f_t$ is incompressible in one direction when $t < \ell_1$ or $t > \ell_2$, $\pi^s$ determines a unique disjoint pair of disks for every value of $s$ for which the graphic for $f \times \pi^s$ is generic.  In particular, Lemma~\ref{grapchicdefinesdiskslem} implies that for every generic $s$, there is a unique maximal interval $(a_s, b_s)$ and a unique defining pair such that for $t \in (a_s, b_s)$, the boundaries of the defining pair are level sets of $\pi^s|_{f_t}$.  The essential level sets of $\pi^s|_{f_t}$ change only when $t$ passes through a value where there is a crossing in $\{t\} \times S^1$.  Thus the circles $\{a_s\} \times S^1$ and $\{b_s\} \times S^1$ pass through crossings in the graphic.

The map $f \times \pi^s$ can fail to be generic for two reasons: the map may fail to be stable (i.e. when it undergoes one of the moves in Figure~\ref{gmovesfig}) or there may be two crossings of the graphic at the same value of $t$.  Generically, each of these will happen for only finitely many values of $s$.

If $f \times \pi^s$ fails to be stable because there is a point in $M$ with a non stable neighborhood (i.e. when the graphic changes by one of the first three types of moves), the graphic changes within a subset $I \times S^1$ for some interval $I \subset [0,1]$.  For moves (1) and (2), there are no crossings in this band, so $I$ is either disjoint from $(a_s,b_s)$ or a proper subset of $(a_s, b_s)$.  For move (3), there is a single crossing in $I \times S^1$ but this crossing cannot be an endpoint of $(a_s,b_s)$ because at least one of the arcs involved corresponds to a pair of pants with a trivial boundary loop in $f_t$. Again $I$ is disjoint from or properly contained in $(a_s,b_s)$.  The crossings created (removed) by moves (4) and (6) cannot be endpoints of $(a_s, b_s)$ because before (after) the move, these crossings don't exist.  

For moves (1),(2),(3),(4) and (6), the move is disjoint from $\{t\} \times S^1$ for some $t \in (a_s, b_s)$ so the level loops of $\pi^s|_{f_t}$ bounding a defining pair do not change.  The disjoint pair with boundaries in level sets of $\pi_s|_{f_t}$ does not change as $s$ passes through $s_0$ so the disjoint pair induced by $\pi_s$ does not change.

The last two cases to deal with are cases in which more than one crossings pass through the same vertical arc of the graphic.  This will only change the induced defining disks if before and after the move, these crossings sit in the vertical loops $a_s \times S^1$ and $b_s \times S^1$.  When the move occurs, this region shrinks down to a single arc, then expands back to an annulus, but with a new pair of loops bounding defining disks as in Figure~\ref{definingchangefig}.
\begin{figure}[htb]
  \begin{center}
  \includegraphics[width=3.5in]{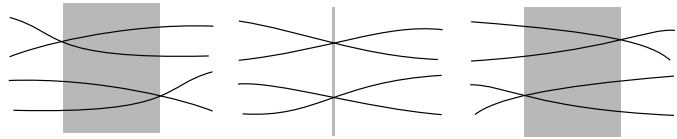}
  \caption{The induced defining disks change when the interval $[a_s,b_s]$ shrinks to a point and the graphic fails to be generic.}
  \label{definingchangefig}
  \end{center}
\end{figure}

To understand how this can happen, we must consider the following situation: for some $t$ and $s$, two level sets of $\pi^s|_{f_t}$ bound a defining pair. Just to the left of $t$ there is a crossing in the graphic that eliminates one of the loops in the pair, and just to the right of $t$ there is a crossing that eliminates the other loop in the pair.

Note that no component of the complement of $\Sigma \setminus (\partial D_1 \cup \partial D_2)$ is a pair of pants.  This implies that the critical points just above and just below one of the loops in the defining pair are distinct from the critical points just above and below the other loop.  Thus the two crossings must involve four distinct edges of the graphic, so move (5) cannot change the pair of defining disks.

The only case in which four distinct edges of the graphic can be involved in the crossing is when $f \times \pi^s$ is stable but not generic because there are two crossings at the same value of $t$.  Each boundary loop of the defining pair sits in a four punctured sphere bounded by essential level loops and each crossing corresponds to replacing the boundary of the defining disk with a new loop in the same four punctured sphere.  The new loops form a new defining pair because after the move, the graphic is generic again.  Thus each boundary loop of the defining pair is replaced by a loop bounding a disk in the opposite handlebody.

Given a pair of defining disks and bounding level sets of some $\pi^s$ as on the left side of Figure~\ref{newdisksfig}, we can find a new pair of defining disks to the right.  A four punctured sphere in $\Sigma$ that contains a loop bounding a disk in $H_1$ contains at most one loop bounding a disk in $H_2$ and vice versa.  Thus the new pair of defining disks shown on the right of the figure is the pair defined by the graphic on the right.
\begin{figure}[htb]
  \begin{center}
  \includegraphics[width=3.5in]{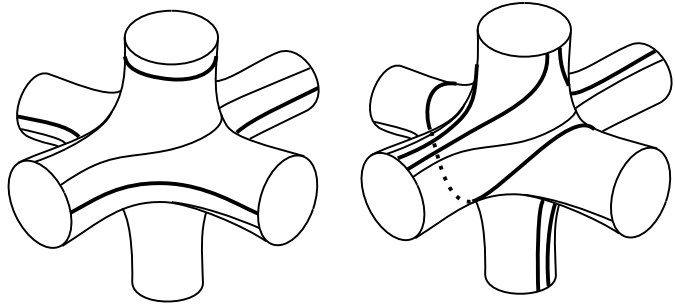}
  \caption{The pairs of defining disks before and after the graphic becomes non-generic.}
  \label{newdisksfig}
  \end{center}
\end{figure}

One can check that the original defining pair shown on the left is taken to the new pair shown on the right by a torus twist along a torus $T$ that intersects $S$ in a single loop and cuts $\Sigma$ into two four punctured spheres.

We have found a sequence of stable functions such that each determines a pair of defining disks for $\Sigma$, the first and last stable functions determine $D_1$,$D_2$ and $D'_1$, $D'_2$, respectively, and the defining disks determined by consecutive stable functions are either isotopic or related by a torus twist along a torus meeting the criteria for Lemma~\ref{twistslem1}.  This defines a sequence of torus twists taking $D_1$,$D_2$ onto $D'_1$,$D'_2$, completing the proof.
\end{proof}

\section{Proof of Theorem~\ref{mainthm}}
\label{mainproofsect}

Recall that $G$ is the subgroup of $Mod(M, \Sigma)$ generated by the elements $\alpha_{ij}$, $\sigma$ and $\tau$ defined above.  To prove Theorem~\ref{mainthm}, we must show that $G$ is equal to $Mod(M,\Sigma)$.

\begin{proof}[Proof of Theorem~\ref{mainthm}]
Let $\phi$ be an element of $Mod(M, \Sigma)$.  If $\phi$ interchanges the handlebodies $H_1$ and $H_2$, then composing $\phi$ with $\sigma$ produces an automorphism that sends $H_1$ to itself.  Since $\sigma$ is an element of $G$, $\sigma \phi$ will be in $G$ if and only if $\phi$ is in $G$.  Thus by replacing $\phi$ with $ \sigma \phi$ if necessary, we will assume that $\phi$ preserves each handlebody.

The image $i(\phi)$ of $\phi$ in $Mod(M)$ is a composition of the images $\{i(\alpha_{ij})\}$.  Thus composing $\phi$ with a sequence of these maps produces an element of $Mod(M,\Sigma)$ that is isotopy trivial in $T^3$.  The maps $\{\alpha_{ij}\}$ are in $G$ so $\phi$ is in $G$ if and only if this composition is in $G$.  We can thus assume $\phi$ is in the kernel of $i$.

Let $T$ be the incompressible torus determined by $D_1^1$ and $D_2^1$.  The disks $\phi(D_1^1)$ and $\phi(D_2^1)$ form a defining pair for $\Sigma$ and determine an incompressible torus $T'$ isotopic to $\phi(T)$.  Because $\phi$ is isotopic to the identity on $T^3$, $T'$ is in fact isotopic to $T$.  Because $D_1^1$, $D_2^1$ and $\phi(D_1^1)$, $\phi(D_2^1)$ determine isotopic tori in $T^3$, Lemma~\ref{twistslem1} implies that there is a sequence of torus twists, along tori that intersect $T$ in a single loop and cut $\Sigma$ into planar pieces, taking $D_1^1$ to $\phi(D_1^1)$ and $D_2^2$ to $\phi(D_2^2)$.  

Any torus twist along a torus that intersects $T$ in a single loop is a conjugate of $\tau$ by elements of $G$, by Lemma~\ref{conjtwistlem}.  Composing $\phi$ by this conjugate of $\tau$ produces an element of $Mod(M,\Sigma)$ that is isotopy trivial on $T^3$ and preserves $D_1^1$ and $D_2^2$.  This conjugate is in $G$ if and only if $\phi$ is in $G$, so we can replace $\phi$ with an element of $Mod(M,\Sigma)$ that is in the kernel of $i$ and takes $D_1^1$ and $D_2^1$ onto themselves.

The complement in $\Sigma$ of $\partial D_1^1 \cup \partial D_1^2$ is a pair of twice punctured tori.  Let $S$ be one of these twice punctured tori.  The automorphisms $\alpha_{23}$, $\alpha_{32}$, $\tau_{12}$ and $\tau_{23}$ take $S$ onto itself and their restrictions to $S$ generate the mapping class group of the twice punctured torus.  (We should only need $\tau_{12}$ and $\tau_{23}$, but for our purposes we can use all four automorphisms.)  Thus there is an element $g$ of $G$ such that $g \circ \phi$ restricts to the identity on $S$.

Each of the loops $\partial D_1^2$ and $\partial D_1^3$ intersects $\partial D_2^1$ in two points, so each loop intersects $S$ in a single properly embedded arc.  Because each loop bounds a disk in $H_1$, it must intersect the twice punctured torus $\Sigma \setminus S$ in a single arc parallel to the arc in $S$. Because $g \circ \phi$ is the identity on $S$, it fixes $\partial D_1^2 \cap S$ and $\partial D_1^3 \cap S$, and therefore also fixes the parallel arcs in $\Sigma \setminus S$.  This implies that $g \circ \phi$ is the identity on $\Sigma \setminus S$ as well as $S$.  Thus $g \circ \phi$ is the identity on all of $\Sigma$ so $\phi = g^{-1} \in G$, completing the proof.
\end{proof}

\bibliographystyle{abbrv}
\bibliography{mcg}

\end{document}